\documentclass[11pt]{amsart}
\usepackage{amsfonts,amsmath,amssymb,epsfig,cite,graphicx,hyperref,
color, esint, fancyhdr, enumerate, latexsym,amsrefs, mathrsfs}

\newtheorem{thm}{Theorem}[section]

\newtheorem{lem}[thm]{Lemma}
\newtheorem{cor}[thm]{Corollary}

\theoremstyle{definition}
\newtheorem{defn}[thm]{Definition}

\newcommand{\ep}{\epsilon}
\newcommand{\mbb}{\mathbb}

\newcommand{\ov}{\overline}

\newcommand{\pa}{\partial}
\newcommand{\mf}{\mathbb}

\newcommand{\Om}{\Omega}
\newcommand{\al}{\alpha}
\newcommand{\be}{\beta}

\newcommand{\la}{\lambda}
\newcommand{\ti}{\tilde}

\renewcommand{\Re}{\operatorname{Re}}

\newcommand{\Ric}{\operatorname{Ric}}

\numberwithin{equation}{section}
\hypersetup{
    colorlinks,
    citecolor=blue,
    filecolor=black, 
    linkcolor=red,
    urlcolor=black
}

\title{Existence of geodesic spirals for the Kobayashi--Fuks metric on planar domains}

\author{Debaprasanna Kar}
\address{Department of Mathematics, Indian Institute of Technology Bombay, Mumbai-400076, India}
\email{dkar@math.iitb.ac.in}

\begin{document}
\maketitle

\begin{abstract}
In this note, we discuss the following problem: Given a smoothly bounded strongly pseudoconvex domain $D$ in $\mathbb{C}^n$, can we guarantee the existence of geodesics for the Kobayashi--Fuks metric which ``spiral around" in the interior of $D$? We find an affirmative answer to the above question for $n=1$ when $D$ is not simply connected.
\end{abstract}

\maketitle


\section{Introduction}

An ingredient in Fefferman's proof of boundary smoothness of biholomorphic mappings between two smoothly bounded strongly pseudoconvex domains are geodesics in the Bergman metric that diverge to the boundary \cite{Fefferman}. This naturally leads to the question if there are geodesics in the Bergman metric that stay within a compact set and this was answered affirmatively by Herbort \cite{Herbort1983} for smoothly bounded strongly pseudoconvex domains having infinitely sheeted universal covers. The purpose of this article is to prove an analog of this result for the Kobayashi--Fuks metric on smoothly bounded non-simply connected planar domains. A geodesic spiral, roughly speaking, is a non-closed geodesic which is ``eventually'' contained inside a compact subset of the underlying domain. Before going into a detailed study of geodesic spirals, let us first briefly recall the construction of the Kobayashi--Fuks metric on bounded domains in $\mbb C^n$. For a more comprehensive study of this metric, the readers may consult the article \cite{Borah-Kar}.

For a bounded domain $D\subset\mf{C}^n$, the space
\[
A^2(D)=\left\{\text{$f: D \to \mf{C}$ holomorphic and $\|f\|^2_D:=\int_{D} \vert f\vert^2 \, dV< \infty$} \right\},
\]
where $dV$ is the Lebesgue measure on $\mf{C}^n$, is a closed subspace of $L^2(D)$, and hence is a Hilbert space. It is called the \textit{Bergman space} of $D$. $A^2(D)$ carries a reproducing kernel $K_D(z,w)$ called the \textit{Bergman kernel} for $D$. Let $K_D(z):=K_D(z,z)$ be its restriction to the diagonal of $D$. It is well known (see \cite{Jarn-Pflug-2013}) that 
\begin{align*}
K_D(z)=\sup\big\{|f(z)|^2: f\in A^2(D), \|f\|_D\leq 1\big\}.
\end{align*}
Since $D$ is bounded, one easily sees that $K_D>0$. It is known that $\log K_D$ is a strictly plurisubharmonic function
and thus is a potential for a K\"ahler metric which is called the Bergman metric for $D$ and is given by
\[
ds^2_{D}(z)=\sum_{\al,\be=1}^n g^{D}_{\al\ov \be} (z) \, dz_{\al}d\ov z_{\be},
\]
where
\[
g^{D}_{\al \ov \be}(z)=\frac{\pa^2 \log K_{D}}{\pa z_{\al} \pa \ov z_{\be}}(z).
\]
We denote
\[
G_{D}(z)=\begin{pmatrix}g^{D}_{\al \ov \be}(z)\end{pmatrix}_{n \times n}.
\]
Recall that the components of the Ricci tensor of $ds^2_{D}$ are defined by
\begin{equation*}\label{Ric-h-ten}
\Ric_{\al \ov \be}^{D}(z)=  - \frac{\pa^2 \log \det G_D}{\pa z_{\al} \pa \ov z_{\be}}(z),
\end{equation*}
and its Ricci curvature is given by
\begin{equation*}\label{Ric_h-curv}
\text{Ric}_{D}(z,v)=\frac{\sum_{\al, \be=1}^n \Ric_{\al \ov \be}^{D}(z) v^{\al} \ov v^{\be}}{\sum_{\al,\be=1}^n g^{D}_{\al\ov \be}(z)v^\al \ov v^\be}.
\end{equation*}
Kobayashi \cite{Kob59} showed that the Ricci curvature of the Bergman metric on a bounded domain in $\mbb C^n$ is strictly bounded above by $n+1$ and hence the matrix
\[
\ti G_{D}(z)=\begin{pmatrix} \ti g^{D}_{\al \ov \be}(z)\end{pmatrix}_{n \times n} \quad \text{where} \quad \ti g^{D}_{\al \ov \be}(z):=(n+1)g^{D}_{\al\ov \be}(z)-\Ric^{D}_{\al\ov \be}(z),
\]
is positive definite (see also Fuks \cite{Fuks66}). Therefore,
\[
d\ti s^2_{D}=\sum_{\al,\be=1}^n \ti g^{D}_{\al \ov \be}(z)\,dz_{\al}d\ov z_{\be}
\]
is a K\"ahler metric with K\"ahler potential $\log (K_D^{n+1} \det G_{D})$. We call this metric the Kobayashi--Fuks metric on $D$. Similar to the transformation rule of the Bergman metric, if $F: D \to D'$ is a biholomorphism, we have
\begin{equation*}\label{tr-kf}
\ti G_{D}(z)= F'(z)^t\, \ti G_{D'}\big(F(z)\big) \ov F'(z),
\end{equation*}
where $F'(z)$ is the complex Jacobian matrix of $F$ at $z$. This implies that $ds^2_{\ti B,D}$ is an invariant metric.

Among some major developments around the Kobayashi--Fuks metric, Dinew \cite{Dinew11} observed that the Kobayashi--Fuks metric plays an important role in the study of the Bergman representative coordinates, a tool introduced by Bergman in his program of generalizing the Riemann mapping theorem to $\mbb C^n, n>1$. Dinew, in his article \cite{Dinew13}, found out a criterion that detects classes of domains which are complete under the Kobayashi--Fuks metric. In a corollary, he showed that on any bounded hyperconvex domain the Kobayashi--Fuks metric is complete, and hence in particular, by a result of Demailly \cite{Dem87}, this metric is complete on any bounded pseudoconvex domain with Lipschitz boundary.  Recently, in \cite{Borah-Kar}, the boundary behavior of the Kobayashi--Fuks metric has been obtained on certain classes of domains by localizing this metric and some of its related curvatures near holomorphic peak points of pseudoconvex domains. In the above article the existence of closed geodesics with a prescribed homotopy class is also discussed for the Kobayashi--Fuks metric. Since the Kobayashi--Fuks metric is closely related to the Bergman metric, it is natural to explore the classical properties that the Bergman metric enjoy in the setting of the Kobayashi--Fuks metric as well.

The results in this article are motivated by a theorem of Herbort  \cite[Theorem~(3.2)]{Herbort1983} on the existence of geodesic spirals for the Bergman metric on strongly pseudoconvex domains in $\mbb C^n$. We will use the asymptotic formula of Fefferman~\cite{Fefferman} and Boutet de Monvel and Sj\"ostrand~\cite{Monvel-Sjostrand}, devised to study the asymptotic behavior of the Begman kernel, as a prime tool for our investigation concerning the geodesic spirals. As we will see in the subsequent sections of this article, the study of the boundary behavior of objects related to the Kobayashi--Fuks metric will help immensely in establishing the existence of geodesic spirals. Let us fix some notations before we state our results.

Throughout this article, ``smoothly bounded'' domains will mean bounded domains having $C^{\infty}$-smooth boundary. For a domain $D$ in dimension one, we denote by
\begin{align*}
ds^2_{D}(z)=g_D(z)|dz|^2 \quad \text{and} \quad d \ti s^2_{D}(z)=\ti g_D(z)|dz|^2
\end{align*}
the Bergman metric and the Kobayashi--Fuks metric on $D$ respectively. We may use the notations $ds^2_{D}$ and $g_D$ (similarly, $d\ti s^2_{D}$ and $\ti g_D$) interchangeably to specify the Bergman metric (similarly, the Kobayashi--Fuks metric) in dimension one. Recall that the K\"ahler potential of the Kobayashi--Fuks metric in dimension $1$ is $\log A_D(z)$ where
\begin{equation*}\label{pot-A}
A_D=K_D^2g_{D}=K_D\tfrac{\pa^2 K_D}{\pa z \pa \ov z}-\tfrac{\pa K_D}{\pa z}\tfrac{\pa K_D}{\pa \ov z}.
\end{equation*}
Thus
\begin{align*}\label{Fuks_tensor}
\ti g_{ D}(z)=\frac{\pa^2 \log A_D}{\pa z \pa \ov z}(z).
\end{align*}

\begin{defn}
Let $(X,h)$ be a complete Riemannian manifold.
\begin{itemize}
\item[(a)] A geodesic $c:\mbb{R}\rightarrow X$, which is not closed, will be called a \textit{geodesic spiral} if there is a compact subset $K$ of $X$ and $t_0\in \mbb{R}$ such that $c|_{[t_0, \infty)}\subset K$.
\item[(b)] Let $c:\mbb R\rightarrow X$ be a non-trivial geodesic and $x_0$ be a point in $X$. If there exist $t_1,t_2\in \mbb R$ with $t_1<t_2$ and $c(t_1)=c(t_2)=x_0$, we will call the segment $c|_{[t_1,t_2]}$ a \textit{geodesic loop} passing through $x_0$.
\end{itemize}
\end{defn}

Here are the main results that we present in this article:

\begin{thm}\label{local minima}
Let $D$ be a smoothly bounded domain in $\mbb C$ and suppose that $\rho$ is a $C^{\infty}$-smooth strictly subharmonic defining function for $D$. Then there exists a positive number $\ep=\ep(D)$ such that for each geodesic $c:\mbb R\rightarrow D$ for the Kobayashi--Fuks metric satisfying $\rho\big(c(0)\big)>-\ep$ and $\left(\rho\circ c\right)'(0)=0$, we must have $(\rho\circ c)''(0)>0$.
\end{thm}

Note that every smoothly bounded domain in $\mbb C$ has a strictly subharmonic defining function of class $C^{\infty}$. To see this consider any bounded planar domain $D$ with smooth boundary $\pa D$, and fix a $C^{\infty}$ defining function $\rho$ for $D$ which is defined on a neighborhood $U$ of $\ov D$. Without loss of generality assume $U$ is bounded. Therefore both the following quantities
\begin{align*} 
\inf_{z\in \ov U}\bigg|\dfrac{\pa \rho}{\pa z}(z)\bigg| \quad \text{and}\quad \sup_{z\in \ov U}\bigg|\dfrac{\pa^2 \rho}{\pa z \pa \ov z}(z)\bigg|
\end{align*}
are finite and attained. Now choose a positive real number $\la$ satisfying
\begin{align*}
\la>\dfrac{\max_{z\in \ov U}\bigg|\dfrac{\pa^2 \rho}{\pa z \pa \ov z}(z)\bigg|}{\min_{z \in \ov U}\bigg|\dfrac{\pa \rho}{\pa z}(z)\bigg|^2}.
\end{align*}
Then one can check that
\begin{align*}
\ti \rho(z):=\dfrac{\exp\big(\la \rho(z)\big)-1}{\la}
\end{align*}
is a $C^{\infty}$ defining function for $D$ which is strictly subharmonic on $U$.

\begin{thm}\label{main thm} 
Let $D$ be a smoothly bounded domain in $\mbb C$ which is not simply connected. Then for every $z_0 \in D$ which does not lie on a closed geodesic, there exists a geodesic spiral for the Kobayashi--Fuks metric passing through $z_0$.
\end{thm}

Although the set of non-simply connected domains in $\mbb C$ is large, the boundedness and boundary smoothness assumptions in the hypothesis of Theorem~\ref{main thm} somehow restricts the class of domains for which the existence of geodesic spirals is exhibited. However, we observe that these conditions can be relaxed if the domain is finitely connected. Indeed, recall that a domain $D\subset \mbb C$ is called \textit{$m$-connected} if its complement in the extended complex plane has $m+1$ components. If, in addition, none of these components is a singleton, then $D$ is called a \textit{non-degenerate $m$-connected domain}. It is known that such a domain is biholomorphic to a smoothly bounded domain, in fact a domain with real analytic boundary (see for example Theorem 2.1 of \cite{Conway1995}). Thus, an immediate consequence of Theorem~\ref{main thm} is the following:

\begin{cor}\label{finitely connected domains}
Let $D$ be a non-degenerate $m$-connected planar domain and $m\geq 1$. Then for every $z_0 \in D$ which does not lie on a closed geodesic, there exists a geodesic spiral for the Kobayashi--Fuks metric passing through $z_0$.
\end{cor}

\textit{Acknowledgements.} The author would like to thank D. Borah for suggesting the problem, as well as for his constant encouragement and guidance. The author is also grateful to Prof. G. Herbort for his kind inputs. Some of the results presented here, especially in Section~\ref{bdry behav}, has benefited from the interactions the author had with Prof. Herbort over emails. Finally, the author thanks the anonymous referee for valuable suggestions for improving the exposition herein, especially for pointing out the result that is presented in Corollary~\ref{finitely connected domains}.

\section{Boundary behavior}\label{bdry behav}

In this section we will use the following result on asymptotic expansion of the Bergman kernel to derive a few boundary estimates for the Kobayashi--Fuks metric.

\begin{thm}[Fefferman~\cite{Fefferman}, Boutet de Monvel and Sj\"ostrand~\cite{Monvel-Sjostrand}]
Given a smoothly bounded strongly pseudoconvex domain $\Om=\{\rho<0\}$ in $\mbb C^n$, there exist functions $\Phi,\Psi\in C^{\infty}(\ov \Om)$ with $\Phi>0$ near $\pa \Om$ such that the Bergman kernel function (on the diagonal) $K_{\Om}$ of $\Om$ is expressed by the formula
\begin{equation*}
K_{\Om}(z)=\dfrac{h_1(z)}{\big(-\rho(z)\big)^{n+1}}
\end{equation*}
with $h_1(z)=\Phi(z)+\big(-\rho(z)\big)^{n+1}\Psi(z)\log\big(-\rho(z)\big)$.
\end{thm}

In the above theorem, and henceforth, $C^{\infty}(\ov \Om)$ denotes the set of $C^{\infty}$-smooth functions on $\Om$ which extends smoothly to a larger open neighborhood containing $\ov \Om$. From this point onwards, $D$ will denote a smoothly bounded domain in $\mbb C$, and $\rho$ will be a strictly subharmonic defining function for $D$. That is,
\begin{align*}
D=\{z\in U: \rho(z)<0\},
\end{align*}
where $U$ is a neighborhood of $\ov D$ and $\rho:U\rightarrow \mbb R$ is a strictly subharmonic function on $U$ of class $C^{\infty}$ with $\pa \rho/\pa z \neq 0$ on the boundary $\pa D=\{\rho=0\}$. We will simply write $g(z)$ for $g_D(z)$ and $\ti g(z)$ for $\ti g_D(z)$ since the domain under consideration is clear. Note that in dimension 1, we have $$h_1(z)=\Phi(z)+\big(-\rho(z)\big)^{2}\Psi(z)\log\big(-\rho(z)\big),$$ and the Bergman kernel $K_D$ is given by
\begin{equation}\label{ker n=1}
K_D(z)=\dfrac{h_1(z)}{\big(-\rho(z)\big)^{2}}.
\end{equation}
Computing the partial derivatives of $K_D$ using formula~(\ref{ker n=1}), one obtains
\begin{align*}
\dfrac{\pa K_D}{\pa z}&=(-\rho)^{-2}\dfrac{\pa h_1}{\pa z}-2 (-\rho)^{-3}h_1 \dfrac{\pa \rho}{\pa z},\quad \text{and}\\
\dfrac{\pa^2 K_D}{\pa z \pa \ov z}&= (-\rho)^{-2}\dfrac{\pa^2 h_1}{\pa z \pa \ov z}-4 (-\rho)^{-3}\Re \bigg(\dfrac{\pa h_1}{\pa z}\dfrac{\pa \rho}{\pa \ov z}\bigg)-2(-\rho)^{-3}h_1\dfrac{\pa^2 \rho}{\pa z \pa \ov z}+6(-\rho)^{-4}h_1\left|\dfrac{\pa \rho}{\pa z}\right|^2.
\end{align*}
Using the above expressions, we compute
\begin{align*}
A_D(z) &=\bigg(K_D\dfrac{\pa^2 K_D}{\pa z \pa \ov z}-\dfrac{\pa K_D}{\pa z}\dfrac{\pa K_D}{\pa\ov z}\bigg)\bigg|_z\\
&=\big(-\rho(z)\big)^{-4}h_1(z)\dfrac{\pa^2 h_1}{\pa z \pa \ov z}(z)-\big(-\rho(z)\big)^{-4}\left|\dfrac{\pa h_1}{\pa z}(z)\right|^2-2\big(-\rho(z)\big)^{-5}h_1^2(z) \dfrac{\pa^2 \rho}{\pa z \pa \ov z}(z)\\
&\quad+2 \big(-\rho(z)\big)^{-6}h_1^2(z) \left|\dfrac{\pa \rho}{\pa z}(z)\right|^2\\
&=:\dfrac{h_2(z)}{\big(-\rho(z)\big)^{6}},
\end{align*}
where 
\begin{align}\label{iden 11}
h_2=(-\rho)^{2}h_1\dfrac{\pa^2 h_1}{\pa z \pa \ov z}-(-\rho)^{2}\left|\dfrac{\pa h_1}{\pa z}\right|^2-2(-\rho)h_1^2 \dfrac{\pa^2 \rho}{\pa z \pa \ov z}+2h_1^2 \left|\dfrac{\pa \rho}{\pa z}\right|^2.
\end{align}
Observe that in our case $h_2(z)>0$ on $D$ as $h_2(z)=\big(-\rho(z)\big)^{6}A_D(z)$, and $$A_D(z)=K_D^2(z)\dfrac{\pa^2 \log K_D}{\pa z\pa \ov z}>0$$ owing to the fact that the Bergman metric is positive definite on bounded domains. Therefore $\log h_2(z)$ is a well defined smooth function on $D$.
 
Let us fix some further notations. We set
\begin{align*}
\mathfrak{h}(z):= \log h_2(z), \quad z\in D.
\end{align*}
For a smooth function $f$, henceforth we will write
\begin{align*}
f_z:=\dfrac{\pa f}{\pa z},\quad f_{\ov z}:=\dfrac{\pa f}{\pa\ov z},\quad f_{z^2}:=\dfrac{\pa^2 f}{\pa z^2},\quad f_{z\ov z}:=\dfrac{\pa^2 f}{\pa z\pa \ov z},\quad f_{z^2\ov z}:=\dfrac{\pa^3 f}{\pa z^2\pa \ov z}, \quad \text{and so on.}
\end{align*}

\begin{lem}\label{asym h}
There exist universal positive constants $S_1,\ldots,S_5$ depending only on $D$ such that
\begin{align*}
\big|\mathfrak{h}_z(z)\big| &\leq S_1,\\
\big|\mathfrak{h}_{z\ov z}(z)\big| &\leq S_2 \big|\log|\rho(z)|\big|+S_3,\\
\big|\mathfrak{h}_{z^2\ov z}(z)\big| &\leq S_4 \big|\rho(z)\big|^{-1}+S_5,
\end{align*}
for $z\in D$.
\end{lem}

\begin{proof}
This result will follow from a bare-hand computation and observing carefully to what extent the logarithm terms in $h_2(z)$ influence the behavior of all its partial derivatives near the boundary. First note that, for any $z_0\in \pa D$,
\begin{align*}
\lim_{z\in D\to z_0} h_2(z)=2\Phi^2(z_0)\bigg|\dfrac{\pa \rho}{\pa z}(z_0)\bigg|^2>0.
\end{align*}
So the function $h_2(z)$ stays away from zero even near the boundary $\pa D$. Therefore both $h_2(z)$ and $\log h_2(z)$ will exhibit similar asymptotic behavior on $D$ along with all their respective partial derivatives.

Denoting $f$ either of the functions $h_1$ or $h_2$, observe that we can write
\begin{align*}
\dfrac{\pa^{m+n}f}{\pa z^m \pa \ov z^n}=\big(\text{smooth part on}\,\ov D\big)+\big(\text{remaining problematic part on}\,\ov D\big)
\end{align*}
for $m,n\in \mbb N\cup \{0\}$. Let us explain the terms smooth part and problematic part on $\ov D$ clearly. The smooth part on $\ov D$ will be the maximal sum of product of several functions, in the complete expansion of $\tfrac{\pa^{m+n}f}{\pa z^m \pa \ov z^n}(z)$ by the product rule (without writing addition of any two functions inside a bracket), such that in each summand, all the individual functions involved in the product are of class $C^{\infty}(\ov D)$. We used the term ``maximal'' in the previous line to ensure that each summand present in the remaining problematic part on $\ov D$ (which is again a product of several functions) will contain at least one function which is not of class $C^{\infty}(\ov D)$. For example, we have
\begin{align*}
h_1(z)&=\Phi(z)+\rho^2(z)\Psi(z)\log\big(-\rho(z)\big),\\
\dfrac{\pa h_1}{\pa z}(z)&=\dfrac{\pa \Phi}{\pa z}(z)+\rho(z)\Psi(z)\dfrac{\pa \rho}{\pa z}(z)+\rho^2(z)\dfrac{\pa \Psi}{\pa z}(z)\log\big(-\rho(z)\big)+2\rho(z)\Psi(z)\dfrac{\pa \rho}{\pa z}(z)\log\big(-\rho(z)\big).
\end{align*}
Therefore the smooth parts on $\ov D$ for the functions $h_1(z)$, $\tfrac{\pa h_1}{\pa z}(z)$ are 
\[
\Phi(z) \quad\text{and} \quad \dfrac{\pa \Phi}{\pa z}(z)+\rho(z)\Psi(z)\dfrac{\pa \rho}{\pa z}(z)
\]
respectively. Similarly, the problematic parts on $\ov D$ for the functions $h_1(z)$, $\tfrac{\pa h_1}{\pa z}(z)$ are 
\[
\rho^2(z)\Psi(z)\log\big(-\rho(z)\big) \quad\text{and} \quad \rho^2(z)\dfrac{\pa \Psi}{\pa z}(z)\log\big(-\rho(z)\big)+2\rho(z)\Psi(z)\dfrac{\pa \rho}{\pa z}(z)\log\big(-\rho(z)\big)
\]
respectively. We will use the following notation,
\begin{align*}
\dfrac{\pa^{m+n}f}{\pa z^m \pa \ov z^n}(z)\prec \al(z) \quad \text{if}\quad \bigg|\text{problematic part in}\,\dfrac{\pa^{m+n}f}{\pa z^m \pa \ov z^n}(z)\bigg|\leq S \big|\al(z)\big|
\end{align*}
near $\pa D$ for some uniform constant $S>0$ . Note that, in the previous example, we have
\begin{align*}
h_1(z)&\prec \rho^2(z)\log\big(-\rho(z)\big),\\
\dfrac{\pa h_1}{\pa z}(z) &\prec \rho(z)\log\big(-\rho(z)\big).
\end{align*}
Furthermore, after computing $\pa^2 h_1/\pa z\pa \ov z$ explicitly using the product rule, one can check that
$$\dfrac{\pa^2 h_1}{\pa z\pa \ov z}(z)\prec \log\big(-\rho(z)\big).$$
Plugging in the expressions of $h_1$, $\pa h_1/ \pa z$, $\pa^2 h_1/\pa z\pa \ov z$ in (\ref{iden 11}), we can check that each summand in the problematic part on $\ov D$ for the function $h_2(z)$ is a function of the form 
\begin{align}\label{iden 12}
a(z)\rho^k(z)\log\big(-\rho(z)\big),\quad \text{with}\,\,\,k\geq 2,
\end{align}
where $a(z)$ is an element of $C^{\infty}(\ov D)$. Therefore,
\begin{align*}
h_2(z)\prec \rho^2(z)\log\big(-\rho(z)\big).
\end{align*}
Expression (\ref{iden 12}), in particular, implies that the problematic part on $\ov D$ for the function $\pa h_2/\pa z$ is a finite sum of the functions of the form
\begin{align*}
b(z)\rho^k(z)\log\big(-\rho(z)\big),\quad \text{with}\,\,\,k\geq 1,
\end{align*}
where $b(z)$ is an element of $C^{\infty}(\ov D)$. This clearly implies
\begin{align}\label{iden 13}
\dfrac{\pa h_2}{\pa z}(z) &\prec \rho(z)\log\big(-\rho(z)\big).
\end{align}
Proceeding in a similar manner, computing the successive partial derivatives of $h_2(z)$ and observing their problematic parts on $\ov D$, we conclude
\begin{align}\label{partial h_2}
\dfrac{\pa^2 h_2}{\pa z\pa \ov z}(z)&\prec \log\big(-\rho(z)\big),\\ \nonumber
\dfrac{\pa^3 h_2}{\pa z^2\pa \ov z}(z)&\prec \big|\rho(z)\big|^{-1}.
\end{align}
The lemma now follows immediately from (\ref{iden 13}) and (\ref{partial h_2}).
\end{proof}

\begin{lem}\label{asym g tilde}
There exist an open neighborhood $V\supset \pa D$ and a universal positive constant $C=C(D)$ such that on $V\cap D$ the following bounds hold:
\begin{itemize}
\item[(i)] $\left|\ti g^{-1} \rho_z\right|\leq C \rho^2$.
\item[(ii)] $\left|\dfrac{1}{\rho^2} \ti g^{-1}\rho_z- \dfrac{1}{6Q}\rho_{z\ov z}^{-1}\rho_z\right|\leq C |\rho|\log \dfrac{1}{|\rho|}$,\\
with $Q(z)=\big(\rho_{z\ov z}(z)\big)^{-1}\big|\rho_z(z)\big|^2$.
\item[(iii)] $\left|\dfrac{1}{\rho^2} \ti g^{-1}|\rho_z|^2- \dfrac{1}{6}\right|\leq C |\rho|\log \dfrac{1}{|\rho|}$.
\end{itemize}
\end{lem}

\begin{proof}
Let us define 
\begin{align}\label{g hat}
\hat g(z):=\big(-\log(-\rho)\big)_{z\ov z}(z)=\dfrac{\rho_{z\ov z}(z)}{-\rho(z)}+\dfrac{\big|\rho_z(z)\big|^2}{\rho^2(z)}.
\end{align}
Since $h_2(z)=\big(-\rho(z)\big)^6 A_D(z)$, we have 
\begin{align}\label{iden 6}
\ti g(z)=6 \hat g(z)+\mathfrak{h}_{z\ov z}(z)
=6\bigg(1+\dfrac{1}{6}\mathfrak{h}_{z\ov z}\hat g^{-1}\bigg)\hat g(z).
\end{align}
Therefore, 
\begin{align}\label{iden 1}
\ti g^{-1}=\dfrac{1}{6}\hat g^{-1}\bigg(1+\dfrac{1}{6}\mathfrak{h}_{z\ov z}\hat g^{-1}\bigg)^{-1}.
\end{align}
Moreover
\begin{align*}
\ti g^{-1}-\dfrac{1}{6}\hat g^{-1} &=\dfrac{1}{6}\hat g^{-1}\bigg[\bigg(1+\dfrac{1}{6}\mathfrak{h}_{z\ov z}\hat g^{-1}\bigg)^{-1}-1\bigg]\\
&=\dfrac{1}{6}\hat g^{-1}\bigg(1+\dfrac{1}{6}\mathfrak{h}_{z\ov z}\hat g^{-1}\bigg)^{-1}\bigg[1-\bigg(1+\dfrac{1}{6}\mathfrak{h}_{z\ov z}\hat g^{-1}\bigg)\bigg]\\
&=-\dfrac{1}{36}\hat g^{-1}\bigg(1+\dfrac{1}{6}\mathfrak{h}_{z\ov z}\hat g^{-1}\bigg)^{-1}\mathfrak{h}_{z\ov z}\hat g^{-1},
\end{align*}
which is same as
\begin{align}\label{iden 4}
\ti g^{-1}=\dfrac{1}{6}\hat g^{-1}-\dfrac{1}{36}\hat g^{-1}\bigg(1+\dfrac{1}{6}\mathfrak{h}_{z\ov z}\hat g^{-1}\bigg)^{-1}\mathfrak{h}_{z\ov z}\hat g^{-1}.
\end{align}
Hence, before studying $\ti g^{-1} \rho_z$, let us first focus on $\hat g^{-1} \rho_z$.

For $z\in D$, using (\ref{g hat}) one can compute
\begin{align}\label{g hat inv}
\hat g^{-1}(z)=\dfrac{\big|\rho(z)\big|^2}{\big|\rho(z)\big|\rho_{z\ov z}(z)+\big|\rho_z(z)\big|^2}&=\big|\rho(z)\big|\bigg(1-\dfrac{\rho^{-1}_{z\ov z}(z)|\rho_z(z)|^2}{|\rho(z)|+\rho^{-1}_{z\ov z}(z)|\rho_z(z)|^2}\bigg)\rho_{z\ov z}^{-1}(z)\\ \nonumber
&=\big|\rho(z)\big|\bigg(1-\dfrac{Q(z)}{|\rho(z)|+Q(z)}\bigg)\rho_{z\ov z}^{-1}(z).
\end{align}
This gives us 
\begin{align}\label{iden 3}
\hat g^{-1} \rho_z=\dfrac{\rho^2}{|\rho|+Q}\rho_{z\ov z}^{-1}\rho_z.
\end{align}
Note that, by the choice of our defining function $\rho$, there exists a constant $c_0=c_0(D)>0$ such that $|\rho|+Q\geq c_0$ on $\ov{D}$. Therefore
\begin{align}\label{iden 2}
\big|\hat g^{-1} \rho_z\big|\leq \dfrac{C_1}{c_0}\rho^2, \quad \text{where}\quad C_1=\sup_{\ov D}\big|\rho_{z\ov z}^{-1}\rho_z\big|.
\end{align}
Now, observe that $\hat g^{-1}$ tends to zero near $\pa D$ not slower than $|\rho|$. This can be seen by (\ref{g hat inv}), as
\begin{align*}
\hat g^{-1}=|\rho|\bigg(1-\dfrac{Q}{|\rho|+Q}\bigg)\rho_{z\ov z}^{-1}\leq |\rho|\rho_{z\ov z}^{-1}\leq \dfrac{1}{c_1} |\rho|,
\end{align*}
with the constant $c_1$ satisfying $\rho_{z\ov z}(z)\geq c_1>0$ on $\ov D$. Again, by Lemma~\ref{asym h}, $\mathfrak{h}_{z\ov z}$ is bounded by $C_2\big|\log|\rho|\big|$ with some unimportant constant $C_2>0$. Hence $\mathfrak{h}_{z\ov z}\hat g^{-1}$ goes to zero not slower than $C_3|\rho|\log \dfrac{1}{|\rho|}$ near the boundary $\pa D$, and thus $\big(1+\tfrac{1}{6}\mathfrak{h}_{z\ov z}\hat g^{-1}\big)^{-1}$ remains bounded. This observation, when paired with equations (\ref{iden 1}) and (\ref{iden 3}), proves (i).

To prove (ii), note that equation~(\ref{iden 4}) implies
\begin{align}\label{iden 5}
\dfrac{1}{\rho^2} \ti g^{-1}\rho_z- \dfrac{1}{6Q}\rho_{z\ov z}^{-1}\rho_z =& \dfrac{1}{6}\bigg(\dfrac{1}{\rho^2}\hat g^{-1}\rho_z-\dfrac{1}{Q}\rho_{z\ov z}^{-1}\rho_z\bigg)\\ \nonumber
&-\dfrac{1}{36}\hat g^{-1}\bigg(1+\dfrac{1}{6}\mathfrak{h}_{z\ov z}\hat g^{-1}\bigg)^{-1}\mathfrak{h}_{z\ov z}\bigg(\dfrac{1}{\rho^2}\hat g^{-1}\rho_z\bigg).
\end{align}
The first term on the right hand side of the above equation is by (\ref{iden 3}) equal to $$\dfrac{1}{6}\dfrac{-|\rho|}{Q\big(|\rho|+Q\big)}\rho_{z\ov z}^{-1}\rho_z,$$
which tends to zero as $|\rho|\rightarrow 0$. The second term on the right hand side in (\ref{iden 5}) also goes to zero, not slower than $|\rho|\log\dfrac{1}{|\rho|}$, because of the following observations we have already made:
\begin{itemize}
\item[$\bullet$] $\hat g^{-1}$ tends to zero near $\pa D$ not slower than $|\rho|$,\\
\item[$\bullet$] $\bigg(1+\dfrac{1}{6}\mathfrak{h}_{z\ov z}\hat g^{-1}\bigg)^{-1}$ is bounded on $D$,\\
\item[$\bullet$] $\mathfrak{h}_{z\ov z}$ is bounded by $C_2\big|\log|\rho|\big|$ for some $C_2>0$, and finally,
\item[$\bullet$] $\dfrac{1}{\rho^2}\hat g^{-1}\rho_z$, by (\ref{iden 2}), is bounded on $D$.
\end{itemize}
This proves (ii).

Next, (iii) follows directly from (ii) with the following observation
\begin{align*}
\dfrac{1}{\rho^2} \ti g^{-1}|\rho_z|^2- \dfrac{1}{6}= \rho_{\ov z}\bigg(\dfrac{1}{\rho^2} \ti g^{-1}\rho_z- \dfrac{1}{6Q}\rho_{z\ov z}^{-1}\rho_z\bigg),
\end{align*}
and the fact that $|\rho_{\ov z}|$ is bounded on $D$.

\end{proof}

\section{Proof of the results}

We need an intermediate lemma to prove Theorem~\ref{local minima}:
\begin{lem}\label{rho compose c}
Let $c:\mbb R\rightarrow D$ be a geodesic for the Kobayashi--Fuks metric. Then for each $t\in \mbb R$ we have
\begin{align}\label{lemma rho of c}
(\rho\circ c)''(t) &=-\dfrac{1}{\rho\big(c(t)\big)}\Re\bigg[\big(\rho\mathfrak{h}_{z^2\ov z}-6\rho_{z^2\ov z}\big)\big(\ti g^{-1}\rho_z\big)\bigg|_{c(t)}\big(c'(t)\big)^2\bigg]\\ \nonumber
&+\dfrac{2}{\rho\big(c(t)\big)}\Re\bigg(\rho_z\big(c(t)\big)c'(t)\bigg)^2-\dfrac{2}{\rho\big(c(t)\big)}\Re \bigg[\big(\mathfrak{h}_{z\ov z}\ti g^{-1}\rho_z\big)\bigg|_{c(t)}\rho_z\big(c(t)\big)\big(c'(t)\big)^2\bigg]\\ \nonumber
&+2\rho_{z\ov z}\big(c(t)\big)\big|c'(t)\big|^2+2\bigg(1-\dfrac{6}{\rho^2}\ti g^{-1}|\rho_z|^2\bigg)\bigg|_{c(t)}\Re\bigg(\rho_{z^2}\big(c(t)\big)\big(c'(t)\big)^2\bigg).
\end{align}
\end{lem}

\begin{proof}
Let us consider the Lagrange function in dimension one
\begin{align*}
\mathcal{L}\big(c,c'\big):=\dfrac{1}{2}\ti g(c)\big(c'\big)^2
\end{align*}
related to the Kobayashi--Fuks metric. One can write the Euler-Lagrange equation in the following complexified form
\begin{align}\label{Euler Lagrange}
\dfrac{d}{dt}\bigg(\dfrac{\pa \mathcal{L}}{\pa c'}\bigg)=\dfrac{\pa \mathcal{L}}{\pa c}.
\end{align}
Equation~(\ref{Euler Lagrange}) implies
\begin{align}\label{iden 7}
-c''=\dfrac{1}{2}\bigg(\ti g^{-1}\dfrac{\pa \ti g}{\pa z}\bigg)(c)\big(c'\big)^2.
\end{align}
To prove the expression in Lemma~\ref{rho compose c}, we start by computing
\begin{align}\label{rho of c}
(\rho\circ c)''(t)=2\Re\bigg(\dfrac{\pa \rho}{\pa z}\big(c(t)\big)c''(t)\bigg)+2\Re\bigg(\dfrac{\pa^2\rho}{\pa z^2}\big(c(t)\big)\big(c'(t)\big)^2\bigg)+2 \dfrac{\pa^2\rho}{\pa z\pa \ov z}\big(c(t)\big)\big|c'(t)\big|^2.
\end{align}
In order to compute $\pa \ti g/\pa z$, we will use the following relation \big(see (\ref{iden 6})\big)
\begin{align}\label{g tilde}
\ti g=\mathfrak{h}_{z\ov z}+6\bigg(\dfrac{\rho_{z\ov z}}{-\rho}+\dfrac{|\rho_z|^2}{\rho^2}\bigg),
\end{align}
and therefore obtain
\begin{align*}
\dfrac{\pa \ti g}{\pa z}=\mathfrak{h}_{z^2\ov z}+6\bigg(\dfrac{2}{\rho^2}\rho_{z\ov z}\rho_z-\dfrac{1}{\rho}\rho_{z^2\ov z}+\dfrac{1}{\rho^2}\rho_{z^2}\rho_{\ov z}-\dfrac{2}{\rho^3}\rho_z|\rho_z|^2\bigg).
\end{align*}
Using this expression in (\ref{iden 7}) one gets
\begin{align}\label{iden 8}
-2c''=\bigg(\mathfrak{h}_{z^2\ov z}-\dfrac{6}{\rho}\rho_{z^2\ov z}\bigg)\ti g^{-1}\big(c'\big)^2+\dfrac{12}{\rho^2}\ti g^{-1}\rho_{z\ov z}\rho_z\big(c'\big)^2+\dfrac{6}{\rho^2}\ti g^{-1}\rho_{z^2}\rho_{\ov z}\big(c'\big)^2\\ \nonumber
-\dfrac{12}{\rho^3}\rho_{\ov z}\ti g^{-1}\big(\rho_z c'\big)^2.
\end{align}
Using the formula \big(which can be obtained from (\ref{g tilde})\big)
\begin{align*}
\rho_{z\ov z}=-\dfrac{\rho}{6}\big(\ti g-\mathfrak{h}_{z\ov z}\big)+\dfrac{|\rho_z|^2}{\rho},
\end{align*}
and substituting the above value of $\rho_{z\ov z}$ in (\ref{iden 8}), we arrive at 
\begin{align*}
-2c''=\bigg(\mathfrak{h}_{z^2\ov z}-\dfrac{6}{\rho}\rho_{z^2\ov z}\bigg)\ti g^{-1}\big(c'\big)^2+\bigg(-\dfrac{2}{\rho}+\dfrac{2}{\rho}\ti g^{-1}\mathfrak{h}_{z\ov z}\bigg)\rho_z \big(c'\big)^2+\dfrac{6}{\rho^2}\ti g^{-1}\rho_{z^2}\rho_{\ov z}\big(c'\big)^2.
\end{align*}
Finally, substituting this expression of $c''$ in equation~(\ref{rho of c}), we obtain the desired formula~(\ref{lemma rho of c}).

\end{proof}

With Lemma~\ref{rho compose c} and the boundary behavior established in Section~\ref{bdry behav}, we now have enough machinery to prove Theorem~\ref{local minima}.

\begin{proof}[Proof of Theorem~\ref{local minima}]
Assume that the assertion is not true. Then for each $k\in \mbb N$, there exists geodesic $c_k$ for the Kobayashi--Fuks metric satisfying 
\begin{align*}
\rho\big(c_k(0)\big)>-\dfrac{1}{k},\quad (\rho\circ c_k)'(0)=0 \quad \text{and}\quad (\rho\circ c_k)''(0)\leq 0.
\end{align*}
Let us denote, for each $k$,
\begin{align*}
a_k:=c_k(0),\quad v_k:=\dfrac{c_k'(0)}{\big|c_k'(0)\big|},\quad \text{and}\quad b_k:=\dfrac{(\rho\circ c_k)''(0)}{\big|c_k'(0)\big|^2}\leq 0.
\end{align*}
Passing to a subsequence, if necessary, we assume that the sequence of points $a_k$ converges to a point $a_0$ on the boundary $\pa D$, and the unit vectors $v_k$ converge to a unit vector $v_0\in \mbb C$. Now, making use of (\ref{lemma rho of c}), we can compute
\begin{align}\label{iden 9}
b_k-\dfrac{2}{\rho(a_k)}\Re\big(\rho_z(a_k)v_k\big)^2=-\Re\bigg[\big(\rho\mathfrak{h}_{z^2\ov z}-6\rho_{z^2\ov z}\big)(a_k)\bigg(\dfrac{1}{\rho}\ti g^{-1}\rho_z\bigg)(a_k)v_k^2\bigg]\\ \nonumber
-2\Re\bigg[\mathfrak{h}_{z\ov z}(a_k)\bigg(\dfrac{1}{\rho}\ti g^{-1}\rho_z\bigg)(a_k)\rho_z(a_k)v_k^2\bigg]+2\rho_{z\ov z}(a_k)|v_k|^2\\ \nonumber
+12\bigg(\dfrac{1}{6}-\dfrac{1}{\rho^2}\ti g^{-1}|\rho_z|^2\bigg)(a_k)\Re\bigg(\rho_{z^2}(a_k)v_k^2\bigg).
\end{align}
Observe that all other terms, except for $2\rho_{z\ov z}(a_k)|v_k|^2$, on the right hand side of equation~(\ref{iden 9}) tend to zero as $a_k$ approaches the boundary point $a_0$, by virtue of Lemma~\ref{asym h} and Lemma~\ref{asym g tilde}. Therefore,
\begin{align}\label{iden 10}
\lim_{k\to \infty}\bigg(b_k-\dfrac{2}{\rho(a_k)}\Re\big(\rho_z(a_k)v_k\big)^2\bigg)=2\rho_{z\ov z}(a_0)|v_0|^2.
\end{align}
Because we have $(\rho\circ c_k)'(0)=0$, one can check $\Re\big(\rho_z(a_k)v_k\big)=0$. Hence the left hand side of (\ref{iden 10}) is a limit of non-positive real numbers, while the right hand side of (\ref{iden 10}) is strictly positive. This is a contradiction, which proves our theorem.

\end{proof}

Theorem~\ref{main thm} will now follow from a more general result of G. Herbort on the existence of geodesic spirals in a complete Riemannian manifold setting. Theorem~\ref{local minima}, in conjuction with the following result of Herbort, immediately proves Theorem~\ref{main thm}.

\begin{lem}[Herbort]\label{Herbort}
Let $(X,h)$ be a complete Riemannian manifold which possesses an infinite sheeted universal cover. Let $x_0$ be a point in $X$ such that there is no closed geodesic passing through $x_0$. If there exists a compact subset $K$ of $X$ such that each geodesic loop passing through $x_0$ is contained in $K$, then there is a geodesic spiral passing through $x_0$.
\end{lem}

For a proof of the above lemma one can refer to \cite{Herbort1983}. Now, in order to prove Theorem~\ref{main thm}, note that the universal cover of $D$ is infinitely sheeted as $D$ is not simply connected. Next, consider
\begin{align*}
\ep_1=\min\big\{\ep(D), -\rho(z_0)\big\}\quad \text{and}\quad K=\bigg\{z\in D: \rho(z)\leq -\dfrac{\ep_1}{2}\bigg\}.
\end{align*}
Theorem~\ref{local minima} then implies that each geodesic loop for the Kobayashi--Fuks metric passing through $z_0$ must be contained inside the compact set $K$. To see this, let us consider any geodesic loop $c|_{[t_1,t_2]}$ passing through $z_0$, i.e. $c(t_1)=c(t_2)=z_0$. If this loop goes outside of $K$, then there would exist $t_*\in (t_1,t_2)$ where the real valued function $\rho\circ c$ attends a local maxima, and that $c(t_*)\in D\setminus K$. Therefore, by the second derivative test $(\rho\circ c)''(t_*)\leq 0$, which contradicts Theorem~\ref{local minima}. So we can now apply Lemma~\ref{Herbort} in our set-up to conclude the existence of a geodesic spiral passing through $z_0$, which proves our result.
\qed\\

\noindent \textbf{Concluding remarks:} Using similar lines of argument employed in this article, I feel one can prove the existence of geodesic spirals for the Kobayashi--Fuks metric on a smoothly bounded strongly pseudoconvex domain, carrying infinitely sheeted universal cover, in higher dimensions as well. For this, one has to carefully define similar functions $h_2(z)$ for domains in $\mbb C^n$ and then study the asymptotic behavior of some of the partial derivatives of $\log h_2(z)$ as was done in Section~\ref{bdry behav}. Since the Kobayashi--Fuks metric is defined using the Ricci tensor of the Bergman metric, which involves computing of the determinant of certain $(n\times n)$ matrix and then obtaining the second order partial derivatives of that determinant, the computations become huge. So the higher dimensional case is beyond the scope of this note. \\

\noindent \textbf{Funding} The research of the author was supported by the Institute Post-doctoral Fellowship program at Indian Institute of Technology Bombay (Institute ID: 20002836).\\
 
\noindent \textbf{Data availability} Data sharing is not applicable to this article as no datasets were generated or analysed during the current study.\\

\noindent \textbf{Conflict of interest} The author declares that he has no conflict of interest.

\end{document}